\documentclass{amsart}
\usepackage{mathrsfs}
\usepackage{mathtools} 
\usepackage{calligra}
\usepackage{amsthm}
\usepackage{amscd}
\usepackage{amssymb}
\usepackage{latexsym}
\usepackage{graphicx}
\usepackage{stmaryrd}
\usepackage{xypic}
\xyoption{curve}
\usepackage{nicefrac}
\usepackage{wasysym}
\usepackage{enumitem}
\usepackage{setspace}
\usepackage{ifthen}
\usepackage{verbatim}
\usepackage{eufrak}
\usepackage{relsize} 
\usepackage{tikz-cd}
\usepackage{stackengine} 
\usepackage{relsize} 
\usepackage[margin=1.5in]{geometry}

\input{notation.tex} 

\newtheorem{thm}{Theorem}[section]
\newtheorem{cor}[thm]{Corollary}
\newtheorem{lem}[thm]{Lemma}
\newtheorem{prop}[thm]{Proposition}

\newtheorem*{lem*}{Lemma}
\newtheorem*{thm*}{Theorem}
\newtheorem*{prop*}{Proposition}

\theoremstyle{definition}

\newtheorem{eg}[thm]{Example}

\newcommand{\XSTREAMS}{\mathscr{U}}
\newcommand{\XSPACES}{\mathscr{T}}
\newcommand{\XPRESETS}{\mathscr{P}}
\newcommand{\XLATS}{\mathscr{K}}

\newcommand{\XMONOIDS}{\mathscr{M}}

\newcommand{\PROPUnderlyingMonoidalPreorderedSets}{Proposition 5.11, \cite{krishnan2009convenient}}
\newcommand{\ThmDiEmbed}{Theorem 2.5, \cite{fernandes2007classification}}

\newcommand{\PropTop}{Proposition 5.8, \cite{krishnan2009convenient}}

\newcommand{\PROPInclusions}{Proposition 5.6, \cite{krishnan2013cubical}}
\newcommand{\PropSD}{Proposition 7.4, 7.5, \cite{krishnan2013cubical}}
\newcommand{\THMPospaces}{Theorem 4.7, \cite{krishnan2009convenient}}
\newcommand{\PROPLocCompactStreams}{Theorem 5.4, \cite{krishnan2009convenient}}
\newcommand{\ThmDHomotopy}{Theorem 7.1, \cite{krishnan2013cubical}}
\newcommand{\THMVanKampen}{Theorem 1, \cite{krishnan2013cubical}}

\newcommand{\THMXClosed}{Theorem 5.12, \cite{krishnan2009convenient}}
\newcommand{\THMQCubical}{Theorem 4.18, \cite{jardine2006categorical}}
\newcommand{\THMTriEquivalence}{Theorem, \cite{jardine2002cubical}}
\newcommand{\CORTriEquivalence}{Theorem, \cite{jardine2002cubical}}
\newcommand{\THMRealizeEquivalence}{Theorem, \cite{jardine2002cubical}}
\newcommand{\THMAMS}{Theorem, \cite{jardine2002cubical}}
\newcommand{\THMTransfer}{Theorem 3.10, \cite{riehl2011algebraic}}
\newcommand{\THMPointwise}{Theorem 4.3, \cite{riehl2011algebraic}}
\newcommand{\THMClassicalEquivalence}{Theorem 4.3, \cite{riehl2011algebraic}}

\newcommand{\LEMQTCocontinuous}{Lemma 7.2, \cite{krishnan2013cubical}}
\newcommand{\LEMSdTri}{Lemma 7.2, \cite{krishnan2013cubical}}
\newcommand{\LEMMincubicalMonics}{Lemma 6.2, \cite{krishnan2013cubical}}
\newcommand{\EGMSpaces}{Examples 2.2 and 3.8, \cite{cole2006mixing}}

\theoremstyle{plain}
\newtheorem*{thm:m-cofibrant}{Theorem {\ref{thm:m-cofibrant}}}

\newcommand{\THMSimplicialClassicalEquivalence}{Theorem \S2.3, \cite{quillen1967homotopical}}
\newtheorem*{thm:simplicial.classical.equivalence}{\THMSimplicialClassicalEquivalence}

\newcommand{\EGSSets}{Examples 3.6 and 4.9, \cite{gambino2017frobenius}}
\newtheorem*{eg:ssets}{\EGSSets}

\newtheorem*{eg:m-spaces}{\EGMSpaces}
\newtheorem*{lem:mincubical.monics}{\LEMMincubicalMonics}
\newtheorem*{thm:x-closed}{\THMXClosed}
\newtheorem*{thm:pointwise}{\THMPointwise}
\newtheorem*{thm:transfer}{\THMTransfer}
\newtheorem*{thm:small-object-argument}{\THMTransfer}

\newtheorem*{lem:sd.tri}{\LEMSdTri}
\newtheorem*{lem:qt-cocontinuous}{\LEMQTCocontinuous}
\newtheorem*{prop:underlying.monoidal.preordered.sets}{\PROPUnderlyingMonoidalPreorderedSets}
\newtheorem*{prop:locally.compact.streams}{\PROPLocCompactStreams}
\newtheorem*{thm:pospaces}{\THMPospaces}
\newtheorem*{thm:diembed}{\ThmDiEmbed}
\newtheorem*{thm:van-kampen}{\THMVanKampen}
\newtheorem*{thm:q-cubical}{\THMQCubical}
\newtheorem*{thm:tri.equivalence}{\THMTriEquivalence}
\newtheorem*{cor:tri.equivalence}{\CORTriEquivalence}
\newtheorem*{thm:realize.equivalence}{\THMRealizeEquivalence}
\newtheorem*{thm:ams}{\THMAMS}
\newtheorem*{eg:mixing}{Example 2.2 from \cite{cole2006mixing}}
\newtheorem*{thm:classical-equivalence}{\THMClassicalEquivalence}
\newtheorem*{prop:hp}{Proposition \ref{prop:hp}}
\newtheorem*{cor:types}{Corollary \ref{cor:types}}
\newtheorem*{prop:extensions}{Proposition \ref{prop:extensions}}

\newtheorem*{cor:algebraic-q-cubical}{Corollary \ref{cor:algebraic-q-cubical}}
\newtheorem*{cor:algebraic-m-spaces}{Corollary \ref{cor:algebraic-m-spaces}}

\newtheorem*{thm:m-streams}{Theorem {\ref{thm:m-streams}}}
\newtheorem*{thm:q-bicubical}{Theorem {\ref{thm:q-bicubical}}}
\newtheorem*{thm:cubical.dihomotopy}{Theorem {\ref{thm:cubical.dihomotopy}}}
\newtheorem*{cor:cubical.dihomotopy}{Corollary {\ref{cor:cubical.dihomotopy}}}

\newtheorem*{prop:cubical.homotopies}{Proposition {\ref{prop:cubical.homotopies}}}
\newtheorem*{prop:nerve.cubcats}{Proposition {\ref{prop:nerve.cubcats}}}
\newtheorem*{prop:kan.cubcat}{Proposition {\ref{prop:kan.cubcat}}}
\newtheorem*{cor:kan.bicubcat}{Corollary {\ref{cor:kan.bicubcat}}}
\newtheorem*{thm:q-fibrant}{Theorem {\ref{thm:q-fibrant}}}
\newtheorem*{thm:q-semifibrant}{Theorem {\ref{thm:q-semifibrant}}}
\newtheorem*{thm:equivalence}{Theorem {\ref{thm:equivalence}}}
\newtheorem*{cor:type-theoretic}{Corollary {\ref{cor:type-theoretic}}}
\newtheorem*{cor:q-equivalences}{Corollary {\ref{cor:q-equivalences}}}
\newtheorem*{cor:m-equivalences}{Corollary {\ref{cor:m-equivalences}}}
\newtheorem*{cor:excision}{Corollary {\ref{cor:excision}}}
\newtheorem*{cor:equivalence}{Corollary {\ref{cor:equivalence}}}
\newtheorem*{cor:cubical.diequivalence}{Corollary {\ref{cor:cubical.diequivalence}}}
\newtheorem*{thm:cartesian.closed.streams}{\THMXClosed}
\newtheorem*{thm:whitehead}{Theorem \ref{thm:whitehead}}
\newtheorem*{prop:topological}{\PropTop}
\newtheorem*{prop:sd}{\PropSD}
\newtheorem*{prop:inclusions}{\PROPInclusions}
\newtheorem*{thm:d-homotopy}{\ThmDHomotopy}

\title{A Hurewicz model structure for directed topology}
\author{Sanjeevi Krishnan and Paige Randall North}

\begin{document}
\begin{abstract}
	This paper constructs an h-model structure for diagrams of streams, locally preordered spaces.
	Along the way, the paper extends some classical characterizations of Hurewicz fibrations and closed Hurewicz cofibrations.
	The usual characterization of classical closed Hurewicz cofibrations as inclusions of neighborhood deformation retracts extends.
	A characterization of classical Hurewicz fibrations as algebras over a pointed Moore cocylinder endofunctor also extends. 
	An immediate consequence is a long exact sequence for directed homotopy monoids, with applications to safety verifications for database protocols.
\end{abstract}
\maketitle
\tableofcontents
\addtocontents{toc}{\protect\setcounter{tocdepth}{1}}

\section{Introduction}
A complex process in nature can be described by a state space equipped with some kind of directionality reflecting the arrow of time.  
The qualitative behavior of such a process is often invariant with respect to \textit{dihomotopy equivalences}, deformations of the state space that respect the given directionality.
A precise definition of \textit{dihomotopy equivalence} depends upon the precise application of interest (cf. \cite{belton2019obstruction,bubenik2009context,fajstrup2006algebraic,gaucher2003model,grandis2003directed,grandis2005shape,goubault2017directed}).
This paper focuses on the simplest kind, a straightforward refinement of classical homotopy equivalence where all homotopies are taken to be \textit{dihomotopies}, or equivalently homotopies \textit{through} directed maps.  
Examples of behavior invariant in this sense include periodicity in certain dynamical systems \cite[Example 2.8]{fajstrup2003dicovering}, and some types of non-determinism in concurrent computations \cite{fajstrup2006algebraic}.
A model structure on a category of directed spaces, where the weak equivalences are dihomotopy equivalences, provides convenient structure for extending classical homotopical constructions for the directed setting. 
An example is an \textit{h-model structure}, a generalization of the Hurewicz model structure on spaces, for some suitable topologically enriched category of directed spaces.

The purpose of this paper is to construct an h-model structure for \textit{streams}, spaces equipped with cosheaves of local preorders.
There exist h-model structures in the literature for other formalisms of directed spaces \cite{gaucher2019six,kahl2006relative}, although one of these formalisms (\textit{pospaces}) is too restrictive to model the states of a looping process and the other two formalisms (\textit{multipointed d-spaces} and \textit{flows}) are not purely topological in the sense that they can be arbitrarily subdivided in time in a way that changes the directed space, even up to dihomotopy equivalence.
Streams, purely topological models of state spaces that can model looping processes, provides a convenient formalism for describing salient directed structure in order-theoretic terms.
For example, it is possible to identify a large class of state spaces where the global preorder together with the topology completely determines the local causal preorders \cite[Lemmas 4.2, 4.4]{krishnan2009convenient}
For another example, the local causal orders on a spacetime define such a cosheaf; in turn, the local causal orders on small enough open neighborhoods \cite[Theorem 1]{malament1977class}, and hence the entire cosheaf of a stream, completely determines the smooth, causal, conformal structure of the spacetime.  
For yet another example, group-freeness of fundamental monoids is a simple dihomotopical constraint on spacetimes and other streams covered by open sets whose local preorders are antisymmetric \cite[Lemma 2.18]{goubault2009covering}.
In practice, the existence of h-model structures on topologically enriched categories formally follows from the commutativity of certain filtered colimits with certain finite limits \cite[Corollary 5.23]{barthel2013construction}.
That kind of commutativity is difficult if not impossible to prove for streams because the cosheaf condition, a preservation of certain colimits, interacts poorly with limits.  

The existence proof for streams and more general $\indexcat{1}$-shaped diagrams instead relies on certain explicit characterizations of the (co)fibrations.
Classical h-fibrations can be characterized as algebras over the underlying pointed endofunctor of a certain monad, defined as a kind of mapping cocylinder based on \textit{Moore paths}  \cite{may1975classifying}.
Recent work extends that monad to other bicomplete topologically enriched categories containing an object that suitably behaves like the non-negative reals \cite{north2017type}.  
This paper recalls that construction and extends the classical characterization of h-fibrations, at least for streams [Theorem \ref{thm:fibrations}].
A characterization of classical closed Hurewicz cofibrations as neighborhood deformation retracts extends [Theorem \ref{thm:NDR}]; the proof mimics classical arguments, but also requires some non-formal properties of streams.
The characterization of cofibrations will turn out to explain why directed spaces in nature, such as spacetimes or state spaces of concurrent programs, almost never decompose into h-homotopy colimits of simpler directed subspaces.
The characterization also underscores the necessity of working in the setting of $\indexcat{1}$-shaped diagrams: it is impossible to bootstrap an h-model structure on streams for, say, based streams in such a way that based streams in nature are both fibrant and cofibrant [Example \ref{eg:hep.failure}].  
The desired existence follows [Theorem \ref{thm:h-streams}].  

Section \S\ref{sec:ditopology} recalls the definition and some properties of streams.  
Section \S\ref{sec:fibrations} defines and characterizes \textit{h-fibrations}.
Section \S\ref{sec:cofibrations} defines and characterizes \textit{h-cofibrations}.
Section \S\ref{sec:h} proves the existence of an h-model structure whose (co)fibrations are the h-(co)fibrations.
Running examples [Examples \ref{eg:spacetimes}, \ref{eg:classifying.streams}, \ref{eg:compactification}, \ref{eg:di-loops}, \ref{eg:based.classifying.streams}, \ref{eg:dihomotopy.monoids}, \ref{eg:group-completion}, \ref{eg:moore.dipaths}, \ref{eg:di-les}, \ref{eg:les}, \ref{eg:semihomological}] apply the main results to understand some basic properties of bigraded \textit{dihomotopy monoids} $\pi_{p,q}X$ of based streams $X$.
Section \S\ref{sec:formal.verification} sketches an example where an associated long exact sequence for $\pi_{p,q}$ streamlines geometric arguments of practical interest on state spaces.  
Proofs often directly cite formal properties of h-(co)fibrations in general bicomplete topologically enriched categories; the reader is referred elsewhere \cite{cole2006many,may2006parametrized,strom1967note} for the general theory.  
The main results of this paper carry over to other purely topological models of directed spaces like \textit{d-spaces} \cite{grandis2003directed} (that are not multipointed and with enrichment induced by forgetting to spaces).

\subsection*{Conventions}
Notate special categories and functors as follows.
\vspace{.1in}\\
\begin{tabular}{rll}
	$[n]$ & ordinal $\{0<1<\cdots<n\}$\\
  $\XPRESETS$ & preordered sets and monotone functions\\
	$\XLATS$ & connected, compact Hausdorff topological lattices and homomorphisms\\
  $\XSPACES$ & weak Hausdorff k-spaces and continuous maps\\
  $\XSTREAMS$ & weak Hausdorff k-streams and stream maps
\end{tabular}
\vspace{.1in}\\
Write $\star$ for a terminal object in a given category.
Let $\id_o$ denote the identity of an object $o$ in a given category.
Write $\pi_o$ for the component of a universal cone in a given category, evident from context, with codomain $o$.
Let $\indexcat{1}$ denote an arbitrary small category.  
The prefix $\indexcat{1}$ will indicate that an object is indexed by $\indexcat{1}$.  
For example, a \textit{$\indexcat{1}$-set} is a $\indexcat{1}$-shaped diagram in the category of sets and a \textit{$\indexcat{1}$-function} is a natural transformation of $\indexcat{1}$-sets.  
Write $[1]$ for the free category on the arrow $0\ra 1$.  
Write $\cat{2}^{\smallcat{1}}$ for the functor category of $\smallcat{1}$-shaped diagrams in $\cat{2}$ and natural transformations between them.  
Given an adjunction $F\dashv G$, the \textit{adjoint} of a morphism of the form $x\ra Gy$ or $Fx\ra y$ will denote the corresponding morphism of the other form; let $\adjoint{(\zeta)}$ denote the adjoint of a morphism $\zeta$ in a given category across a given adjunction.
The phrases \textit{left lifting property}, \textit{right lifting property} will be respectively abbreviated \textit{LLP}, \textit{RLP}.
Let $\R_+$ denote the subspace $[0,\infty)\subset\R$.


\section{Streams}\label{sec:ditopology}
This section recalls some of the theory of \textit{streams}; the reader is referred elsewhere for the point-set theory \cite{krishnan2009convenient}, including comparisons with other formalisms for directed spaces, and the associated homotopy theory \cite{krishnan2015cubical}.  
A \textit{circulation} on a space $X$ is a function
$$\leqslant:U\mapsto\;\leqslant_U$$
assigning to each open subset $U\subset X$ a preorder $\leqslant_U$ on $U$ such that $\leqslant$ sends the union of a collection $\mathcal{O}$ of open subsets of $X$ to the preorder with smallest graph containing the graph of $\leqslant_U$ for each $U\in\mathcal{O}$ \cite{krishnan2009convenient}.
A \textit{stream} is a space equipped with a circulation on it \cite{krishnan2009convenient}.

\begin{eg}
  \label{eg:initial.circulations}
  Every space admits an \textit{initial circulation} $\leqslant$ defined by
  \begin{equation*}
    x\leqslant_Uy\iff x=y\in U
  \end{equation*}
\end{eg}

\begin{eg}
	\label{eg:spacetimes}
	For each $n$-spacetime $M$, the function
  $$U\;\mapsto\;\leqslant_U$$
	assigning to each open subspacetime $U\subset M$ the causal order $\leqslant_U$ on $U$ defines a circulation on $M$.  
	In particular, the circle $\sphere^1$ will henceforth be regarded as a stream $\sphere^{(1,0)}$ by equipping it with the circulation defined by a time-orientation on $\sphere^1$.  
	On a sufficiently fine open cover, the circulation on $\sphere^1$ assigns each open subset a total order.
\end{eg}

A continuous map $f:X\ra Y$ of streams is a \textit{stream map} if $f(x)\leqslant_Uf(y)$ whenever $x\leqslant_{f^{-1}U}y$ \cite{krishnan2009convenient}.
A \textit{k-space} if $X$ is a colimit of compact Hausdorff spaces in the category of spaces and continuous maps.
Similarly, a \textit{k-stream} is a colimit of compact Hausdorff streams in the category of streams and stream maps \cite{krishnan2009convenient}.
The underlying space of a k-stream is a k-space because the forgetful functor from streams and stream maps to spaces and continuous maps is cocontinuous \cite[Proposition 5.8]{krishnan2009convenient}.
A space $X$ is \textit{weak Hausdorff} if images of compact Hausdorff spaces in $X$ are Hausdorff.

\begin{prop:locally.compact.streams}
  Locally compact Hausdorff streams are weak Hausdorff k-streams.
\end{prop:locally.compact.streams}

Let $\XSPACES$ denote the complete, cocomplete, and Cartesian closed \cite{mccord1969classifying} category of weak Hausdorff k-spaces and continuous maps between them.
Let $\XSTREAMS$ denote the category of weak Hausdorff k-streams and stream maps.
Redefine space and stream, like elsewhere (cf. \cite{krishnan2009convenient, may1999concise}), to means objects in the respective categories $\XSPACES$ and $\XSTREAMS$.
Let $\XLATS$ denote the connected compact Hausdorff topological lattices whose underlying spaces are connected and continuous lattice homomorphisms between them.
Such objects will henceforth be regarded as streams, as noted in the following special case of a more general observation \cite[Lemmas 4.2, 4.4 and Example 4.5]{krishnan2009convenient}.

\begin{thm}
  \label{thm:pospaces}
	There exists a full and faithful embedding
	$$\XLATS\ra\XSTREAMS$$
	naturally sending a topological lattice $L$ to its underlying space equipped with the unique circulation sending $L$ to the given partial order $\leqslant_L$ on $L$.
\end{thm}

\begin{eg}
	\label{eg:classifying.streams}
	Let $B\smallcat{1}$ denote the geometric realization of the simplicial nerve of a small category $\smallcat{1}$.  
	The construction $B$ preserves finite products and hence induces the structure of a topological lattice on the compact Hausdorff, connected space $BL$ natural in finite lattices $L$.
	In this manner, $B[n]$ naturally defines a stream [Theorem \ref{thm:pospaces}].
  Redefine $B\smallcat{1}$ as
  $$B\smallcat{1}=\colim_{[n]\ra\smallcat{1}}B[n]$$
	in $\XSTREAMS$.  
	This \textit{classifying stream} $B\smallcat{1}$ of a small category $\smallcat{1}$ encodes the directionality implicit in the arrows of $\smallcat{1}$ as the circulation on $B\smallcat{1}$.  
\end{eg}

The \textit{forgetful functor} $\XSTREAMS\ra\XSPACES$ lifts topological constructions in the following sense.

\begin{prop:topological}
  The forgetful functor makes $\XSTREAMS$ topological over $\XSPACES$.
\end{prop:topological}

In other words, each class of continuous maps $f_i:X\ra Y_i$ from a space $X$ to streams $Y_i$ induces a terminal circulation on $X$ making the $f_i$'s stream maps $X\ra Y_i$.
Equivalently and dually, each class of continuous maps from streams to a fixed space induces a suitably initial circulation on that space.
In particular, the forgetful functor $\XSTREAMS\ra\XSPACES$ creates limits and colimits.
The reader is referred elsewhere \cite{borceux1994handbook} for the basic theory of categories topological over other categories.
The \textit{forgetful functor} $\XSTREAMS\ra\XPRESETS$ to the category $\XPRESETS$ of preordered sets and monotone functions is the functor naturally sending a stream $X$ to its underlying set equipped with the preorder that the circulation $\leqslant$ on $X$ assigns to $X$ itself.

\begin{prop:underlying.monoidal.preordered.sets}
	The forgetful functor $\XSTREAMS\ra\XPRESETS$ is Cartesian monoidal.
\end{prop:underlying.monoidal.preordered.sets}

The forgetful functor $\XSTREAMS\ra\XPRESETS$ preserves those colimits of streams which are colimits of underlying sets \cite[Lemma 3.18]{krishnan2009convenient}; a special case is the following observation.

\begin{lem}
  \label{lem:colimit.global.preorder}
  The forgetful functor $\XSTREAMS\ra\XPRESETS$ preserves \ldots
  \begin{enumerate}
    \item \ldots coproducts; and
    \item \ldots quotients of streams by equivalence relations with closed graphs
  \end{enumerate}
\end{lem}

Say that a stream map $f:X\ra Y$ \textit{pushes forward} the circulation on $X$ to the circulation on $Y$ if $x\leqslant_{f^{-1}U}y$ whenever $f(x)\leqslant_Uf(y)$ for all choices of $x,y\in X$ and open $U\subset Y$.  

\begin{lem}
  \label{lem:topological.tensors}
	Consider a stream $X$ and space $Y$.  
	The stream map
	\begin{equation*}
		\amalg_{y\in Y}X\xra{\amalg_y\id_X\times(\{y\}\ira Y)}X\times Y
	\end{equation*} 
	pushes forward the circulation on its domain to the circulation on its codomain.  
\end{lem}
\begin{proof}
	Take open substreams $U\subset X$ and $V\subset Y$.  
	For $x_1,x_2\in U$ and $y_1,y_2\in V$,
	\begin{align*}
		      (x_1,y_1)\leqslant_{U\times V}(x_2,y_2)
		&\iff x_1\leqslant_Ux_2,\;y_1\leqslant_Vy_{2}
		\\
  	&\iff x_1\leqslant_Ux_2,\;y_1=y_2
		\\
		&\iff (x_1,y_1)\leqslant_{\amalg_{y\in V}U}(x_2,y_2)
	\end{align*}
	with the first line by the forgetful functor $\XSTREAMS\ra\XPRESETS$ Cartesian monoidal [\PROPUnderlyingMonoidalPreorderedSets], the second line by the initiality of the circulation on $Y$, and the third line by the forgetful functor $\XSTREAMS\ra\XPRESETS$ preserving coproducts [Lemma \ref{lem:colimit.global.preorder}].
\end{proof}

A \textit{substream} of a stream $Z$ is a stream $Y$ for which inclusion of underlying sets defines a stream map $Y\ra Z$ such that every stream map $X\ra Z$ whose image is a subset of $Y$ corestricts to a stream map $X\ra Y$. 
Every subset of a stream has the structure of a substream, where the topology is the k-ification of the subspace topology, by the forgetful functor $\XSTREAMS\ra\XSPACES$ topological.  
The circulation of a substream is generally difficult to ascertain; an exception is when the substream defines an open subset.  

\begin{eg}
  An open substream is an open subspace with a restricted circulation.
\end{eg}

Let $X\ira Y$ denote an inclusion of a substream $X$ into a stream $Y$.  
More generally let $X\ira Y$ denote a $\indexcat{1}$-stream map $X\ra Y$ which, up to isomorphism of $\indexcat{1}$-streams evident from context, is an objectwise inclusion of streams.

\begin{thm:x-closed}
  The category $\XSTREAMS$ is Cartesian closed.
\end{thm:x-closed}

\begin{cor}
	The category $\XSTREAMS^{\indexcat{1}}$ is complete, cocomplete, and Cartesian closed.
\end{cor}

Henceforth $\XSTREAMS^{\indexcat{1}}$ will be regarded as Cartesian monoidal. 
For each $\indexcat{1}$-stream $X$, write $X^{(-)}$ for the right adjoint to $X\times-$.
A \textit{based stream} $(X,x)$ is a stream $X$ equipped with a distinguished point $x\in X$, regarded as the $[1]$-stream $\{x\}\ira X$; a \textit{based stream map} is a $[1]$-stream map of based streams.  
Henceforth regard a stream as a constant $\indexcat{1}$-stream and a $\indexcat{1}$-space as a $\indexcat{1}$-stream with objectwise initial circulations.  

\begin{eg}
	\label{eg:compactification}
	Let $\I^{(1,0)}$ denote the unit interval with the usual ordering.
  Let
	$$\I^{(p,q)}=(\I^{(1,0)})^p\times\I^q\quad\sphere^{(p,q)}=\quotient{\I^{(p,q)}}{\partial\I^{p+q}},$$
	where $\partial\I^n=\{(x_1,\ldots,x_n)\in\I^n\;|\;\prod_ix_i(1-x_i)=0\}$.  
	In a sense that can be made precise, $\sphere^{(p,q)}$ is the terminal compactification of the $(p+q)$-dimensional ordered vector space whose positive cone is the product $\R_{\geqslant 0}^p$; the quotiented point will thus be denoted by $\infty$.
	The underlying space of $\sphere^{(p,q)}$ is the $(p+q)$-sphere.
\end{eg}

\begin{eg}
	\label{eg:di-loops}
	Let $\Omega^{(p,q)}_\star(X,x)$ denote the based mapping substream
  \begin{equation}
		\label{eqn:di-loops}
		\Omega^{(p,q)}_\star(X,x)=(X,x)^{(\sphere^{(p,q)},\infty)}
	\end{equation}
	by regarding based streams as $[1]$-streams.
	In other words, (\ref{eqn:di-loops}) is the stream of all based $(p+q)$-fold loops directed in its first $p$ suspension coordinates.
\end{eg}

In particular, $\XSTREAMS^{\indexcat{1}}$ is bicomplete $\XSPACES$-enriched with enrichment and cotensor defined by restrictions of the closed structure and tensor defined by restriction of binary products.  
The \textit{dihomotopy} relation can be defined in terms of this enrichment.  
Fix $\indexcat{1}$-stream maps $f,g$ as in the left of the diagrams
\begin{equation*}
  \xymatrix@C=3pc{
		**[l]Y\ar[r]^-{f}\ar[d]_-{Y\times(\{0\}\ira\I)}
    & **[r]Z
    \\
      **[l]Y\times\I\ar@{.>}[ur]|{h}
			& **[r]Y\ar[l]^-{Y\times(\{1\}\ira\I)}\ar[u]_-{g}
  }\quad
  \xymatrix@C=3pc{
		**[l]X\times\I\ar[r]^{\pi_X}\ar[d]_-{i\times\I}
    & **[r]X\ar[d]^-{fe=ge}
    \\
      **[l]Y\times\I\ar[r]_-{h}
    & **[r]Z
  }
\end{equation*}
Write $h:f\sim g$ for a dotted $\indexcat{1}$-stream map making the left of the diagrams commute.
For a $\indexcat{1}$-stream map $i:X\ra Y$, write $h:f\sim g$ relative $i$ if the right diagram commutes.

\begin{eg}
	\label{eg:based.classifying.streams}
  For a monoid $M$, the image of the stream map
	$$B(0\ra M):\star\ra BM$$
	naturally turns the classifying stream $BM$ [Example \ref{eg:classifying.streams}] into a \textit{based classifying stream} $B_\star M$.
\end{eg}

Call a $\indexcat{1}$-stream map $f:X\ra Y$ \textit{h-acyclic} if there exists a $\indexcat{1}$-stream map $g:Y\ra X$ with $gf\sim\id_X$ and $fg\sim\id_Y$. 
An \textit{h-equivalence} of $\indexcat{1}$-streams is an h-acyclic $\indexcat{1}$-stream map.  
The working definition of a dihomotopy equivalence in this paper is an h-equivalence of $\indexcat{1}$-streams.

\begin{eg}
	\label{eg:dihomotopy.monoids}
	Define a set $\pi_{p,q}(X,x)$, natural in based streams $(X,x)$, by
	$$\pi_{p,q}(X,x)=\pi_0\Omega^{(p,q)}_\star(X,x),$$
	where $\pi_0$ denotes the ordinary path-component construction.
	Equivalently, $\pi_{p,q}(X,x)$ is the set of all $\sim$-classes of based stream maps $(\sphere^{(p,q)},\infty)\ra(X,x)$ from the stream $\sphere^{(p,q)}$ based at $\infty$ [Example \ref{eg:compactification}] .
  This set \ldots
  \begin{enumerate}
  \item \ldots defines a monoid for $p+q>0$ induced by the co-H structure on based spheres
	\item \ldots is commutative for $p+q>1$ by an Eckmann-Hilton argument
	\item \ldots is a group for $p>1$ or $q>0$ because commuting suspension coordinates gives inverses in the first case and reversing a map gives inverses in the second case.
  \end{enumerate}
  The construction $\pi_{p,q}$ unifies some invariants in the literature.  
	The monoid $\pi_{1,0}X$ is called the \textit{fundamental monoid} of $X$ \cite{fajstrup2006algebraic}.  
	More generally, the groups $\pi_{n,0}X$ have been introduced previously \cite{grandis2002directed} as \textit{higher homotopy monoids}, at least with respect to a homotopy relation that coincides with h-homotopy on examples of interest \cite[Theorem 7.1 ]{krishnan2015cubical}.  
  The groups $\pi_{0,n}X$ are the homotopy groups of the underlying based space of $X$. 
	Stream maps $\sphere^{(p,q+i)}\ra\sphere^{(p+i,q)}$ defined by identity functions induce natural transformations $\pi_{p+i,q}\ra\pi_{p,q+i}$.  
\end{eg}

\begin{eg}
	\label{eg:group-completion}
	Directed simplicial approximation \cite[Theorem 8.1]{krishnan2015cubical} gives
	$$\pi_{1,0}B_\star M=M$$
	for each discrete monoid $M$. 
	The natural homomorphism $\pi_{1,0}B_\star M\ra\pi_{0,1}B_\star M$ is then group-completion $M\ra M[M^{-1}]$.  
\end{eg}

\section{Fibrations}\label{sec:fibrations}
Call a $\indexcat{1}$-stream map $f$ an \textit{h-fibration} if $f$ has the RLP against $X\times(\{0\}\ira\I)$ for each $\indexcat{1}$-stream $X$.
The purpose of this section is to characterize the h-fibrations of $\indexcat{1}$-streams as algebras over the underlying pointed endofunctor of a \textit{Moore path monad} $\Gamma$, defined as follows.  
Define stream $\Pi X$ and stream maps $\mathrm{cod}_{\Pi X}$, $\mathrm{dom}_{\Pi X}$, $\mathrm{id}_{\Pi X}$ by the following commutative diagram below, where the outer right rectangle is a pullback diagram defining $\Pi X$ as a pullback, natural in streams $X$.
\begin{equation*}
		\begin{tikzcd}
       X
		 & \Pi X
			   \ar{rrrr}[above]{\mathrm{cod}_{\Pi X}}
				 \ar[dd,hookrightarrow]
				 \ar{l}[above]{\mathrm{dom}_{\Pi X}}
     &
     &
     &
     & X
		     \ar{dd}[right]{X^{\R_{+}\ra\star}}
     \\
     &
		 &
		 & X
		     \ar{ull}[description]{\mathrm{id}_{\Pi X}}
				 \ar{urr}[description]{\id_X}
				 \ar{dll}[description]{X^{\R_{+}\ra\star}\times(\{0\}\ira\R_{+})}
     \\
       X^{\R_{+}}
			   \ar{uu}[left]{X^{\{0\}\ira\R_{+}}}
		 & X^{\R_{+}}\times\R_{+}
		     \ar{rrrr}[below]{\adjoint(X^{\max:\R_{+}^2\ra\R_{+}})}
				 \ar{l}[below]{\pi_{X^{\R_{+}}}}
		 &
		 &
     &
     & X^{\R_{+}}
	 \end{tikzcd}
\end{equation*}

In other words, $\Pi X$ is a stream of all pairs $(\zeta,t)$ of stream map $\zeta:\R_{+}\ra X$, with $\R_{+}$ regarded as a stream with initial circulation, and $t\in\R_{+}$ with $\zeta$ constant on $[t,\infty)$.  
Thus $\Pi X$ is the classical \textit{Moore path space} (cf \cite{barthel2013construction}) of the underlying space of $X$ equipped with a natural circulation.  
This Moore path space underlying $\Pi X$ is the morphism part of a \textit{Moore path category}, a topological category with identity, codomain, and domain structure maps respectively given by $\mathrm{id}_{\Pi X}$, $\mathrm{cod}_{\Pi X}$, and $\mathrm{dom}_{\Pi X}$.

\begin{prop}
	\label{prop:composition}
	There exists a stream map 
	$$\circ_{\Pi X}:(\Pi X)\times_{\mathrm{cod}_{\Pi X},\mathrm{dom}_{\Pi_X}}(\Pi X)\ra\Pi X$$
	natural in streams $X$, defining the composition operation of a category internal to $\XSTREAMS$ with object stream $X$, morphisms stream $\Pi X$, and identity, domain, and codomain structure maps given by $\mathrm{id}_{\Pi X},\mathrm{cod}_{\Pi X},\mathrm{dom}_{\Pi X}$.
\end{prop}

In other words, $X$ and $\Pi X$ are respectively the object and morphisms streams of a category internal to $\XSTREAMS$; the composite $(\zeta_2,t_2)\circ_{\Pi X}(\zeta_1,t_1)=(\zeta_1*\zeta_2,t_1+t_2)$, where $\zeta_1*\zeta_2$ informally is the map $\R_+\ra X$ which first executes the path $\zeta_1$ on $[0,t_1]$, then executes the path $\zeta_2$ on $[t_1,t_1+t_2]$, and then stays constant on $[t_1+t_2,\infty)$.  
This operation, continuous, defines the composition of a classical Moore path category (cf \cite{barthel2013construction}). 
The main task of the proof is to demonstrate that this operation defines a stream map.

\begin{proof}
	For $t\in\R_{+}$, define continuous maps
  \begin{align*}
    i_t&:\{t\}\ira\R_{+}\\
    r_t=\min(-,t)&:\R_{+}\ra[0,t]\\
		\delta_{-t}&:\{0\}\ira[0,t]\\
		\delta_{+t}&:\{t\}\ira[0,t]
	\end{align*}

	Let $a,b\in\R_{+}$.
  Consider the solid morphisms in the diagram
  \begin{equation*}
	  \begin{tikzcd}
			  X^{[0,a]}\times X^{[0,b]}
				\ar{d}[description]{(X^{r_a}\times i_{a})\times(X^{r_b}\times i_{b})}
				\ar[dr,phantom,"I"]
			& X^{[0,a]}\times_{X^{\delta_{+a}},X^{\delta_{-b}}}X^{[0,b]}
				\ar{rr}[above]{\cong}
				\ar[d,dotted] 
				\ar[l,hookrightarrow]
			& \ar[d,phantom,"II"]
			& X^{[0,a+b]}
			  \ar[d,dotted]
				\ar[r,equals]
				\ar[dr,phantom,"III"]
			& X^{[0,a+b]}
			  \ar{d}[description]{X^{r_{a+b}}\times i_{a+b}}
			\\
			  (X^{\R_{+}}\times\R_{+})^2
			& \Pi X\times_{\mathrm{cod}_{\Pi X},\mathrm{dom}_{\Pi X}}\Pi X
		  	\ar[dotted]{rr}[below]{\circ_{\Pi X}}
				\ar[l,hookrightarrow]
			& \;
			& \Pi X
			\ar[r,hookrightarrow]
			& X^{\R_{+}}\times\R_{+}
 	  \end{tikzcd}
  \end{equation*}
	where $\cong$ denotes the isomorphism induced from the identification $[0,a]\cup_{1,0}[0,b]\cong[0,a+b]$ induced from inclusion $[0,a]\ira[0,a+b]$ and the embedding $x\mapsto x+a:[0,b]\ira[0,a+b]$.
	There exist dotted vertical functions, stream maps by universal properties of substreams and unique by monicity of the bottom left and right arrows, making I and III commute.
	There exists a bottom dotted horizontal continuous map, defined by composition of paths in classical Moore path categories, making II commute in $\XSPACES$.

	The coproduct over all $a,b$ of the leftmost vertical arrow, and hence also the leftmost dotted arrow, are objectwise bijective stream maps pushing forward objectwise circulations of their domains onto objectwise circulations of their codomains [Lemma \ref{lem:topological.tensors}].
	It follows that $\circ_{\Pi X}$ is a stream map because the middle top horizontal arrow is a stream map for each $a,b$. 
	The maps $\circ_{\Pi X},\id_{\Pi X},\mathfrak{cod}_{\Pi X},\mathfrak{dom}_{\Pi X}$ are the respective composition, identity, domain, and codomain operations for a category internal to $\XSPACES$, the classical Moore path category.
	Hence those stream maps satisfy the requisite associativity and unitality properties for the composition, domain, and codomain operations of a category internal to $\XSTREAMS$.  
\end{proof}

\begin{eg}
  \label{eg:moore.dipaths}
  \textit{Moore dipaths} on a stream $X$ can be defined as a subcategory
  $$\Tau X\subset\Pi X,$$
  with domain and codomain stream maps $\mathfrak{dom}_{\Tau X},\mathfrak{cod}_{\Tau X}:\Tau X\ra X$ defined by replacing mapping streams from the undirected reals with mapping streams from the reals equipped with its usual ordering.
\end{eg}

More generally, let $\Pi X$ and $\mathrm{cod}_{\Pi X}$, $\mathrm{dom}_{\Pi X}$, $\mathrm{id}_{\Pi X}$, $\circ_{\Pi X}$ denote the induced constructions of $\indexcat{1}$-stream and $\indexcat{1}$-stream maps, natural in $\indexcat{1}$-streams $X$.
These structure maps turn $\Pi X$ into the morphism $\indexcat{1}$-stream of a category internal to $\XSTREAMS^{\indexcat{1}}$ by naturality.
A construction like $\Pi$ formally yields an associated algebraic weak factorization system (cf. \cite{north2017type,van2012topological}), detailed as follows.  
Define $\indexcat{1}$-stream $X\times_f\Pi Y$ and $\indexcat{1}$-stream maps $\mathrm{dom}_{\Pi Y}^*f,f^*\mathrm{dom}_{\Pi Y},Lf,\Gamma f$ by the following commutative diagrams, natural in $\indexcat{1}$-stream maps $f:X\ra Y$, in which the outer rectangle in the left diagram is a pullback diagram.
  \begin{equation*}
		\begin{tikzcd}
       X\times_{f}\Pi Y
			 \ar{rr}[above]{\mathrm{dom}_{\Pi Y}^*f}
			 \ar{dd}[left]{f^*\mathrm{dom}_{\Pi Y}}
     &
		 & {\Pi Y}
		   \ar{dd}[right]{\mathrm{dom}_{\Pi Y}}
     \\
		   \;
     & X
		   \ar{ul}[description]{Lf}
		   \ar{dl}[description]{\id_X}
	     \ar{ur}[description]{\mathrm{id}_{\Pi Y}f}
		   \ar{dr}[description]{f}
     \\
		   X
			 \ar{rr}[below]{f}
     &
     & Y
	  \end{tikzcd}
	  \quad
		\begin{tikzcd}
			X
			\ar[ddr,phantom,near start,"\eta_f"]
			  \ar{dd}[left]{f}\ar{r}[above]{Lf} & X\times_f\Pi Y\ar{dd}[right]{\mathrm{dom}^*_{\Pi Y}f}\ar{ddl}[description]{\Gamma f}
			\\
			& \;
			\\
			Y & \Pi Y\ar{l}[below]{\mathrm{cod}_{\Pi Y}}
		\end{tikzcd}
  \end{equation*}
	Thus $\Gamma$ will be regarded as an endofunctor on $(\XSTREAMS^{\indexcat{1}})^{[1]}$ pointed by the unit $\eta$ whose components are defined by the commutative triangle above.

\begin{prop}
  \label{prop:monad}
  Commutative diagrams of the following form
  \begin{equation*}
      \begin{tikzcd}
			  	X\times_f\Pi Y\times_{\mathrm{cod}_{\Pi Y},\mathrm{dom}_{\Pi Y}}\Pi_Y
	          \ar[r]
						\ar{d}[left]{\Gamma^2f}
					&  X\times_f\Pi Y\ar{d}[right]{\Gamma f}
        \\
				  Y
					\ar[r,equals]
				& Y
			\end{tikzcd}
		\end{equation*}
		where the top horizontal stream map is naturally induced by $\circ_{\Pi Y}:\Pi Y\times_{\mathrm{cod}_{\Pi Y},\mathrm{dom}_{\Pi Y}}\Pi_Y$, define a multiplication $\mu:\Gamma^2\ra\Gamma$ turning $\Gamma$ into a monad.
\end{prop}

The proof relies on the fact that the restriction and corestriction of $\Gamma$ to a pointed endofunctor on $\XSTREAMS^{[1]}$ underlies a monad with monad multiplication defined as the restriction and corestriction of the multiplication in the statement of the proposition \cite{barthel2013construction}.

\begin{proof}
	The forgetful functor $T:\XSTREAMS^{[1]}\ra\XSPACES^{[1]}$ is faithful and the composite $TF$ of $T$ with its left adjoint $F$ is the identity $\id_{\XSPACES^{[1]}}$.  
	In the case  $\indexcat{1}=\star$, $T\mu_F$ is a multiplication turning the pointed endofunctor $T\Gamma F$ into a monad on $\XSPACES^{[1]}$ and therefore $\mu$ turns $\Gamma$ into a monad.
	The case for general $\indexcat{1}$ follows by naturality.  
\end{proof}

More generally, $\Gamma$ will denote the induced monad on $(\XSTREAMS^{\indexcat{1}})^{[1]}$.   

\begin{thm}
  \label{thm:fibrations}
	The following are equivalent for a $\indexcat{1}$-stream map $f$.
  \begin{enumerate}
    \item $f$ is an h-fibration
    \item $f$ underlies an algebra over the pointed endofunctor $\Gamma$
	\end{enumerate}
\end{thm}

The proof of (2)$\implies$(1) uses a formal characterization of h-fibrations in a bicomplete $\XSPACES$-enriched category as algebras over a certain pointed endofunctor $N$ \cite[Proposition 2.5]{cole2006many}, constructed in the proof for the particular setting of $\indexcat{1}$-streams, on the associated arrow category.  

\begin{proof}
	Take $f:X\ra Y$.  
	\vspace{.1in}
  \textit{(1)$\implies$(2)}:
  Assume (1).  Consider the solid commutative diagram
	\begin{equation*}
		\begin{tikzcd}
			\{0\}
			  \ar[rr,hookrightarrow] 
					\ar[d,equals]
		    	&&
		     {[0,1]}
				  \ar[rr,hookrightarrow] 
					\ar[d,equals]
			&
			& {[0,2]}
			    \ar[rr,hookrightarrow] 
					\ar{d}[description]{i_1}
			&
			& {[0,3]}
			    \ar[rr,hookrightarrow] 
					\ar{d}[description]{i_2}
			&
			& \cdots
			&
			& \R_{+}
			  \ar[dotted]{d}[right]{i}
      \\
      \{0\}
			  \ar[rr,hookrightarrow] 
					\ar[d,equals]
			&&
			  \I
				  \ar[d,equals]
				  \ar{rr}[description]{\I\times(\{0\}\ira\I)} 
			&
			& \I^2
			    \ar{rr}[description]{\I^2\times(\{0\}\ira\I)} 
				  \ar{d}[description]{v\mapsto v_1+v_2}
			&
			& \I^3
			    \ar{rr}[description]{\I^3\times(\{0\}\ira\I)} 
				  \ar{d}[description]{v\mapsto v_1+v_2}
			&
			& \cdots
			&
			& \I^\omega
			  \ar[dotted]{d}[right]{r}
			\\
			\{0\}
			  \ar[rr,hookrightarrow] 
			  &&
		    {[0,1]}
				  \ar[rr,hookrightarrow] 
			&
			& {[0,2]}
			    \ar[rr,hookrightarrow] 
			&
			& {[0,3]}
			    \ar[rr,hookrightarrow] 
			&
			& \cdots
			&
			& \R_{+}
	  \end{tikzcd}
	\end{equation*}
	of spaces and continuous maps defined by $i_n(x)=(1,1,\ldots,1,x-n+1)$ for $n-1\leqslant x\leqslant n$.  
	In each column, the bottom solid vertical arrow is a retraction to the top solid vertical arrow. 
	Thus the vertical arrows induced dotted continuous maps $i,r$ between the transfinite composites of the rows with $r$ a retraction to $i$.
	Let $F$ be the endofunctor 
	$$F=(X \times_f\Pi Y)\times-:\XSTREAMS^{\indexcat{1}}\ra\XSTREAMS^{\indexcat{1}},$$
	cocontinuous by $\XSTREAMS^{\indexcat{1}}$ Cartesian closed and therefore preserving transfinite composites.
	Then $f$ has the RLP against $F(\{0\} \hookrightarrow \I^\omega)$ because it has the RLP against the image of each of the middle horizontal arrows under $F$.
	Thus $f$ has the RLP against $F(\{0\} \hookrightarrow \R_{+})$ by $Fr$ a retraction to $Fi$.  
	Then there exists a dotted $\indexcat{1}$-stream map $\ell$ making the diagram
	\begin{equation*}
	\begin{tikzcd}[column sep=huge]
		  X 
			  \ar{rrrr}[above]{\id_X}
			  \ar{drr}[description]{Lf} 
			  \ar{ddd}[left]{Lf} 
	  &  
		&
		&
		& X \ar{ddd}[right]{f}
	\\
	  &
	  & X\times_f\Pi Y
		  \ar{urr}[description]{\pi_X}
	  	\ar{d}[left]{F(\{0\}\ira\R_{+})}
	\\
	  &
    & X\times_f\Pi Y\times\R_+
			 \ar[dotted]{uurr}[description]{\ell}
			 \ar{drr}[description]{\adjoint(\id_{Y^{\R_{+}}})(\pi_{Y^{\R_+}}\times\R_+)}
	\\
      X\times_f\Pi Y
			\ar{rrrr}[below]{\Gamma f}
			\ar{urr}[description]{(\id_{X\times_f\Pi Y})\times\pi_{\R_+}}
	&
	&
	&
	& Y,
  \end{tikzcd}
  \end{equation*}
	commute.
  The composite $\ell\circ(\id \times\pi_{\R_{+}})$ gives the desired algebra structure for $f$. 

	\vspace{.1in}
  \textit{(2)$\implies$(1)}:
	Define $Nf,\eta_f$, natural in $f$, by the commutative diagram
  \begin{equation}
	  \label{eqn:pushout-products}
		\begin{tikzcd}
			  X\times_fY^{\I}\ar[rr]\ar{dd}[left]{Nf}
      &
			& Y^{\I}\ar{dd}[right]{Y^{\{1\}\ira\I}}
      \\
			& X
			  \ar{ul}[description]{\eta_f}
				\ar{dl}[description]{\id_X}
				\ar{ur}[description]{Y^{(\{1\}\ira\I)}f}
				\ar{dr}[description]{f}
      \\
		  	X\ar{rr}[below]{f}
      &
			& Y,
		\end{tikzcd}
  \end{equation}
	whose outer square is a pullback diagram.
	Thus $N$ defines a pointed endofunctor 
	$$N:(\XSTREAMS^{\indexcat{1}})^{[1]}\ra(\XSTREAMS^{\indexcat{1}})^{[1]}.$$
	
	Consider the solid $\indexcat{1}$-streams in the diagram
  \begin{equation*}
	  \begin{tikzcd}
			  Y^{\I}
				\ar[dotted]{d}[left]{i_f}
				\ar[r,equals]
			& Y^{\I}
		  	\ar{d}[right]{Y^{\min(-,1):\R_{+}\ra\I}\times(\{1\}\ira\R_{+})}
			\\
			  \Pi Y
			\ar[r,hookrightarrow]
			& Y^{\R_{+}}\times\R_{+}
 	  \end{tikzcd}
  \end{equation*}
	There exists a dotted $\indexcat{1}$-function, unique by the right vertical arrow monic, hence a $\indexcat{1}$-stream map $i_f$ by the bottom horizontal arrow an inclusion of substreams and natural in $f$ by unicity, making the entire diagram commute.
	Applying the pullback functor $f^*$ to $i_f$ defines the $f$-component of a map $N\ra\Gamma$ of pointed endofunctors.
	If (2), then $f$ is an $N$-algebra and therefore (1) \cite[Proposition 2.5]{cole2006many}.
\end{proof}

\begin{eg}
	\label{eg:di-les}
  For each pair $(Y,X)$ of streams and $x\in X$, relative variants 
	\begin{align*}
	     \pi_{p+1,q}^{-}(Y,X,x)
	  &= \pi_{p,q}\Tau(Y,X,x)\\
		&= \pi_{0,q}\Tau(\Omega^{(p,0)}_\star Y,\Omega^{(p,0)}_\star X,\{\sphere^{(p,0)}\ra\{x\}\})
	\end{align*}
	where $\Tau(Y,X,x)$ denotes the subcategory of $\Tau Y$ consisting of all Moore dipaths [Example \ref{eg:moore.dipaths}] starting at $x$ and ending at a point in $X$, fit inside sequences of pointed sets below in which functions strictly to the left of $\pi^-_{p,1}(Y,X,x)$ are monoid homomorphisms:
  $$\cdots\ra\pi_{1,q}(X,x)\ra\pi_{1,q}(Y,x)\ra\pi^{-}_{1,q}(Y,X,x)\ra\pi_{0,q}(X,x)\ra\pi_{0,q}(Y,x)$$
  These sequences are generally not even exact but are at least chain complexes (in the sense that composites of successive functions are constant) due to an observation elsewhere \cite[\S 3.5]{grandis2002directed}.
\end{eg}

\section{Cofibrations}\label{sec:cofibrations}
Call a $\indexcat{1}$-stream map $f$ an \textit{h-cofibration} if $f$ has the LLP against $E^{\{0\}\ira\I}$ for each $\indexcat{1}$-stream $E$.

\begin{thm}
  \label{thm:NDR}
	The following are equivalent for a $\indexcat{1}$-stream map $e:X\ra Y$.
  \begin{enumerate}
    \item $e$ has the LLP against all h-acyclic h-fibrations
    \item $e$ is an h-cofibration
    \item There exist dotted arrow $u$ making the square
    \begin{equation}
      \label{eqn:urysohn}
      \begin{tikzcd}
        X\ar{r}[above]{e}\ar[d] & Y\ar[dotted]{d}[right]{u}\\
	\{0\}\ar[r,hookrightarrow] & \I
      \end{tikzcd}
    \end{equation}
		a pullback square and a stream map $f\sim\id_Y$ relative $e$ such that for each $\indexcat{1}$-object $o$, $f_o(y)\in X(o)$ whenever $u_o(y)<1$
  \end{enumerate}
\end{thm}

The proof of (1)$\implies$(2) follows from the formal fact that $o^{\{0\}\ira\I}$ is an h-fibration for all objects $o$ in a bicomplete $\XSPACES$-enriched model category $\modelcat{1}$ \cite[Lemma 4.2.4 (i)]{may2006parametrized}.
The proof of (2)$\iff$(3) uses the established case where $f$ is a continuous map of spaces \cite[Theorem 2]{strom1967note}, the characterization of h-cofibrations between spaces as neighborhood deformation retracts, and mimics a classical argument \cite[proof of Theorem 4]{strom1967note}.
The proof of (3)$\implies$(1) mimics another classical argument \cite[proof of Theorem 3]{strom1967note}, but also uses the initiality of the circulation on $\I$.
The proof of (3)$\implies$(2) uses the facts that an h-cofibration of spaces is a topological embedding \cite[Theorem 1]{strom1967note} and that the inclusion of a mapping cylinder $M(e:X\ra Y)$ into $Y\times\I$ admits a strong deformation retraction if $e$ is an h-cofibration of spaces \cite[Lemma before Theorem 3]{strom1967note}.

\begin{proof}
	Let $o$ denote a $\indexcat{1}$-object.
  Let $T$ denote the forgetful functor
  $$T:\XSTREAMS\ra\XSPACES.$$

  Define $Me$ and $j$ by the commutative diagram whose outer square is a pushout square.
  \begin{equation}
	  \label{eqn:pushout-products}
		\begin{tikzcd}
		    X\times\I\ar[rr]\ar{dr}[description]{e\times\I}
      &
			& Me\ar{dl}[description]{j}
      \\
		  & Y\times\I
      \\
		  	X\ar[ur]\ar{uu}[left]{X\times(\{0\}\ira\I)}\ar{rr}[below]{e}
      &
			& Y\ar{uu}[right]{Y\times(\{1\}\ira\I)}\ar{ul}[description]{Y\times(\{1\}\ira\I)}
		\end{tikzcd}
  \end{equation}

  \vspace{.1in}
  \textit{(1)$\implies$(2)}:
	Assume (1).  
	Then $e$ has the LLP against $Z^{\{0\}\ira\I}$, an h-acyclic h-fibration \cite[Lemma 4.2.4 (i)]{may2006parametrized}, for each $\indexcat{1}$-stream $Z$.

  \vspace{.1in}
  \textit{(2)$\implies$(3)}:
  Assume (2).
  Then $j$ admits a retraction $r$ \cite[Proposition 2.5]{cole2006many}.
	Define $f,u,h:\id_Y\sim f$ relative $e$ by the commutative diagrams
  \begin{equation*}
    \xymatrix{
			**[l]Y\ar[r]^f\ar[d]_{\adjoint(r)}\ar[dr]|{\adjoint(h)} & **[r]Y=Y^{\{1\}}
      \\
			**[l](Me)^{\I}\ar[r]_{(\pi_Yj)^{\I}} & **[r]Y^{\I}\ar[u]_{Y^{\{1\}\ira\I}}
    }\quad
    \xymatrix{
      **[l]Y\ar[r]^u\ar[d]_{\adjoint(r)} & **[r]\I
      \\
			**[l](Me)^{\I}\ar[r]_{(\pi_{\I}j)^{\I}} & **[r]\I^{\I}\ar[u]_{\zeta\mapsto\sup\;\!\!_{t\in\I}|t-\zeta(t)|}
    }
  \end{equation*}

  For each $o$ and $y\in Y(o)$ with $u_o(y)<1$,
  \begin{align*}
		          |1-(\pi_{\I}j)_o(r_o(y,1))|<1
							&\implies (\pi_{\I}j)_o(r_o)(y,1)>0\\
    &\implies j_o(r_o(y,1))\in X(o)\times\I\\
    &\implies (\pi_Yj)_o(r_o(y,1))=f_o(y)\in X(o)
  \end{align*}
	Therefore $f$ restricts and corestricts to a retraction, from the objectwise closed substream $Y\times_u[0,\half]$ of $Y$, to $e$.
	Thus $e$ can be taken to be an inclusion of a substream with objectwise closed image.  

	Fix $o$.
  For each $y\in u_o^{-1}(0)$, $r_o(y,t)\in X(o)\times\I$ for all $t>0$ and hence also for $t=0$ by $X(o)\times\I$ closed in $Y(o)\times\I$.
  Thus $u_o^{-1}(0)\subset X(o)$.
  Conversely for each $x\in X(o)$, $r_o(x,t)=(x,t)$ for all $t\in\I$ and hence $u_o(x)=0$.
  Thus $X(o)\subset u_o^{-1}(0)$.

  Thus $e=u^{*}(\{0\}\ira\I)$ as functions and hence as stream maps by $X$ a substream of $Y$.
  Hence (3).

  \vspace{.1in}
  \textit{(3)$\implies$(2)}:
  Posit $u,f$ and $h:\id_X\sim f$ relative $e$ as in (3).

  Note $Te=T(u^*\{0\}\ira\I)=(Tu)^*(\{0\}\ira\I)$ and $Th:\id_{TY}\sim Tf$ relative $Te$ [\PropTop]. 
  The continuous map $Te$, having objectwise closed image by $\{0\}$ closed in $\I$, is an objectwise Hurewicz cofibration \cite[Theorem 2]{strom1967note} and can therefore be taken to be a objectwise inclusion of spaces \cite[Theorem 1]{strom1967note}.
  Then $j$ admits a retraction \cite[Lemma before Theorem 3]{strom1967note} and hence can also be taken to be an objectwise inclusion of spaces.

  Fix $o$.
  Let $U$ denote an open subset of $Y(o)$.
  Let $V$ denote an open subset of $\I$; for each $V$, let $V_0=V\cap\{0\}$.
  The topology of $(Me)(o)$ has as a basis sets of the form $(Me)(o)\cap(U\times V)$ by $(Me)(o)$ a subspace of $Y(o)\times\I$.
  Consider sets of the form
  \begin{align}
	   &\;\Gamma_{Y(o)\times\I}((Me)(o)\cap(U\times V))\label{eqn:ndr:substream}\\
		 &\;\Gamma_{Y(o)\times V_0}((Y(o)\cap U)\times V_0)\amalg\Gamma_{X(o)\times\I}((X(o)\cap U)\times V)\label{eqn:ndr:coprod}\\
		 &\;\Gamma_{(Me)(o)}((Me)(o)\cap(U\times V))\label{eqn:ndr:mapping},
  \end{align}
  where $\Gamma_Z(W)$ denotes the graph of the preorder that the circulation on the substream $W$ of a stream $Z$ assigns to $W$ itself.
	Firstly, (\ref{eqn:ndr:substream})=(\ref{eqn:ndr:coprod}) [Lemma \ref{lem:topological.tensors}].
	Secondly, (\ref{eqn:ndr:coprod})=(\ref{eqn:ndr:mapping}) because: the circulation on $(Me)(o)$ makes (\ref{eqn:ndr:mapping}) the smallest graph of a preorder containing (\ref{eqn:ndr:coprod}) [Lemma \ref{lem:colimit.global.preorder}
] and (\ref{eqn:ndr:coprod}) defines a circulation on $(Me)(o)$.
  Thus (\ref{eqn:ndr:substream})=(\ref{eqn:ndr:mapping})

  Thus $Me$, whose objectwise circulations agree with the objectwise circulations of a substream of $Y\times\I$ on bases and hence everywhere, is a substream of $Y\times\I$.
  Hence $h$ corestricts to a retraction to $j$.
  Hence (2) \cite[Proposition 2.5]{cole2006many}.

  \vspace{.1in}
	\textit{(2)$\implies$(1)}:
	Assume (2).  
  Then $j$ is an h-acyclic h-cofibration \cite[Proposition 2.10]{cole2006many}.
  For clarity, let $X'=Me$ and $Y'=Y\times\I$.
  There exist retraction $r$ to $j$ and $h:\id_{Y'}\sim jr$ relative $j$ \cite[Lemma 4.2.5 (i)]{may2006parametrized}.
	There exists $u:Y'\ra\I$ with $(\{0\}\ira\I)^*u=j$ by (2)$\implies$(3).  
  It suffices to show $j$ has the LLP against all h-fibrations \cite[Lemma 4.2.4 (ii)]{may2006parametrized}.
  Let $s$ be the continuous map $\I^2\ra\I$ defined by
  $$s(\varepsilon_1,\varepsilon_2)=\min(1,\varepsilon_1^{-1}\varepsilon_2).$$

  Consider the left of the commutative diagrams
  \begin{equation*}
    \xymatrix{
      **[l] X'\ar[r]^{f'}\ar[d]_{j} & **[r] Z'\ar[d]^g\\
      **[l] Y'\ar[r]_{f''}\ar@{.>}[ur]|{f'''} & **[r]Z''
    }\quad
    \xymatrix@C=3pc{
			**[l]Y'\ar[rrrr]^{f'r}\ar[d]|{Y\times(\{0\}\ira\I)} & & & & **[r] Z'\ar[d]^g\\
      **[l]Y'\times\I
			\ar[r]_-{((\id_{Y'}\times u)\times\I)}
        \ar@{.>}[urrrr]|{j'}
      & Y'\times\I^2
			\ar[r]_{Y'\times s}
      & Y'\times\I
        \ar[r]_{h}
      & Y'
        \ar[r]_{f''}
      & **[r]Z''
    }
  \end{equation*}
  of solid arrows with $g$ an h-fibration.
  There exists a dotted arrow $j'$ making the right diagram commute by $g$ an h-fibration.
  Then $f'''=j'(\id_{Y'}\times u)$ makes the left diagram commute.
\end{proof}

The pullback criterion for h-cofibrations highlights a practical difference between classical homotopy and dihomotopy.
Take $\indexcat{1}=\star$.
In the case where $i$ is an inclusion of spaces in nature (the circulations on $X,Y$ both initial, the topology of $Y$ normal), Urysohn's Lemma implies the existence of a pullback square (\ref{eqn:urysohn}).
In the case where $i$ describes an inclusion of state streams in nature (there exist $x\in X$ and $y\in Y-X$ with $x\leqslant_Yy$ or $y\leqslant_Yx$), the triviality of $\leqslant_{\I}$ precludes the existence of a pullback square (\ref{eqn:urysohn}).

\begin{eg}
	\label{eg:hep.failure}
  An inclusion of streams of the form
	$$\star\ira\sphere^{(p,q)}$$
	is not an h-cofibration if $p>0$ because $\infty\leqslant_{\sphere^1}x\leqslant_{\sphere^1}\infty$ for all $x\in\sphere^1$.  
	Thus the isomorphism type of $\pi_{p,q}(X,x)$ for $p>0$ is generally dependent on the choice of basepoint $x$, even on a path-connected stream $X$.  
\end{eg}

\section{Model structure}\label{sec:h}
The classical Hurewicz model structure on $\XSPACES$ thus extends to the following model structure on $\indexcat{1}$-streams.  

\begin{thm}
  \label{thm:h-streams}
	There exists a model structure on $\XSTREAMS^{\indexcat{1}}$ in which \ldots
  \begin{enumerate}
		\item\label{item:h-equivalences} weak equivalences are the h-equivalences
    \item\label{item:h-fibrations} fibrations are the h-fibrations
    \item\label{item:h-cofibrations} cofibrations are the h-cofibrations
  \end{enumerate}
  This model structure is both proper and Cartesian monoidal.
	The forgetful functor $\XSTREAMS^{\indexcat{1}}\ra\XSPACES^{\indexcat{1}}$ is both a left and right Quillen maps from this model structure to an h-model structure on $\indexcat{1}$-spaces.
\end{thm}

The proof hinges on the construction of an (h-acyclic h-cofibration, h-fibration) factorization system.  
The argument that the left side of this system satisfies the requisite lifting properties mimics a proof for the classical setting \cite[proof of Corollary 3.12]{barthel2013construction}.

\begin{proof}
	In the functorial factorization $f=(\Gamma f)(Lf)$, $\Gamma f$ is an algebra over the pointed endofunctor $\Gamma$ [Proposition \ref{prop:monad}] and hence an h-fibration [Theorem \ref{thm:fibrations}].  
	It suffices to show that $Lf$ is an h-acyclic h-cofibration.
	A proper model structure, Cartesian monoidal \cite[Theorem 2.7]{schwanzl2002strong}, satisfying (\ref{item:h-equivalences}), (\ref{item:h-fibrations}), and (\ref{item:h-cofibrations}) would then exist \cite[Theorem 4.3.1]{may2006parametrized}.
	The last sentence would follow from the forgetful functor $\XSTREAMS^{\indexcat{1}}\ra\XSPACES^{\indexcat{1}}$ a $\XSPACES$-enriched left and right adjoint.

  Therefore consider the diagram
  \begin{equation*}
    \begin{tikzcd}
        X\times_f\Pi Y\times\I
          \ar[rrrrrrr,dotted]
					\ar[d,hookrightarrow]
      &
      &
      &
      &
      &
      &
      & X\times_f\Pi Y
			    \ar[d,hookrightarrow]
      \\
        X\times Y^{\R_{+}}\times\R_{+}\times\I
				\ar{rrrrrrr}[below]{X\times(\adjoint(Y^{s\pi'}:Y^{\R_{+}\ra\R_{+}\times\R_{+}\times\I})\times s\pi_{\R_+\times\I})}
      &
      &
      &
      &
      &
      &
      & X\times Y^{\R_{+}}\times\R_{+}
    \end{tikzcd}
  \end{equation*}
	where $s:\R_{+}\times\I\ra\R_{+}$ is multiplication and $\pi':\R_{+}\times\R_{+}\times\I\ra\R_{+}\times\I$ is projection onto first and third factors.
	There exists a dotted $\indexcat{1}$-function, a $\indexcat{1}$-stream map and hence $(Lf)r\sim\id_{X\times_Y\Pi Y}$ relative $Lf$, by the right vertical arrow a $\indexcat{1}$-stream embedding, making the diagram commute.

	Let $r=f^*\mathrm{dom}_{\Pi Y}$.  
  Define the $\indexcat{1}$-stream map $u$ by the commutative rectangle
  \begin{equation*}
    \begin{tikzcd}
        X\times_f\Pi Y\ar{rr}[above]{u}\ar{d}[left]{\mathrm{dom}^*_{\Pi Y}f} 
			& 
			& \I\\
			  \Pi Y\ar{rr}[below]{\pi_{\R_{+}}} 
			& 
			& \R_{+}\ar{u}[right]{\min(1,-)}
    \end{tikzcd}
  \end{equation*}
	Then $u^*(\{0\}\ira\I)=Lf$ as $\indexcat{1}$-functions and hence as $\indexcat{1}$-streams maps because $Lf$, a section to $r$, is a $\indexcat{1}$-stream embedding.
  Thus $Lf$ is an h-cofibration [Theorem \ref{thm:NDR}], h-acyclic by $(Lf)r\sim\id_{X\times_Y\Pi Y}$ and $r(Lf)=\id_{X}$.
\end{proof}

This h-model structure facilitates the construction of dihomotopy invariants with desired formal properties.  
The following relative variants of $\pi_{p,q}$ illustrates this convenience.  

\begin{eg}
	\label{eg:les}
	The construction $\pi_{p,q}$ fits inside an exact sequence
	$$\cdots\ra\pi_{p,1}(X,x)\ra\pi_{p,1}(Y,x)\ra\pi_{p,1}(Y,X,x)\ra\pi_{p,0}(X,x)\ra\pi_{p,0}(Y,x)$$
of pointed sets, where the functions strictly to the left of $\pi_{p,1}(Y,X,x)$ are monoid homomorphisms, by letting $\pi_{p,q+1}(Y,X,x)$ be $\pi_{p,q}$ applied to the h-homotopy pullback of $X\ira Y\la\{x\}$ based at $x$, or equivalently $\pi_{p,0}$ applied to the h-homotopy pullback of $\Omega^{(0,q)}_\star(X,x)\ira\Omega^{(0,q)}_\star(Y,x)\ila\{\sphere^{(0,q)}\ra\{x\}\}$; concretely, $\pi_{p,q}(Y,X,x)=\pi_{p,q+1}\Pi(Y,X,x)$, where $\Pi(Y,X,x)$ is the subcategory of $\Pi Y$ consisting of all Moore paths that start at a point in $X$ and end at $x$.  
\end{eg}

\begin{eg}
	\label{eg:semihomological}
	The construction $B_\star$ [Example \ref{eg:based.classifying.streams}] induces a full and faithful embedding
	$$\XMONOIDS\ira h\XSTREAMS^{[1]}$$
	from the category $\XMONOIDS$ of monoids and monoid homomorphisms to the homotopy category $h\XSTREAMS^{[1]}$ because it admits a retraction $F\mapsto\pi_{1,0}(F(1)/\im\,F(0\ra 1))$ [Example \ref{eg:dihomotopy.monoids}] by an application of simplicial approximation for directed topology \cite[Theorem 8.14]{krishnan2015cubical}.  
	While groups are often studied in terms of classical homotopy invariants on their based classifying spaces (eg. reduced group cohomology), monoids can instead be studied in terms of more refined h-homotopy invariants on their based classifying streams.  
	In this manner, the h-homotopy theory of based streams provides a setting to do homological algebra for monoids (cf. \cite{connes2017homological,patchkoria2000homology,ionescu1998parity}).  
\end{eg}

\section{Verification}\label{sec:formal.verification}
This section suggests a general strategy for state space analysis.  
Translate temporal properties of interest in terms of dihomotopy invariants like $\pi_{*,*}$ [Example \ref{eg:dihomotopy.monoids}] that are bigraded in a temporal degree and a spatial degree.  
Use long exact sequences for such invariants [Example \ref{eg:les}] that vary in the spatial degree, constructed with minimal technical fuss by means of the h-model structure, to get spatio-temporal obstructions to those properties.
Then use some sort of geometric argument to transport those obstructions along degree comparison maps, like $\pi_{p+i,q}\ra\pi_{p,q+i}$, to purely temporal obstructions.  
Then translate these more interpretable obstructions into concrete properties of the system. 
We illustrate this strategy below for a particular class of complex systems.
For convenience, we suppress notation for the basepoint when denoting a based (pair of) stream(s).

\subsection{Concurrent systems}
Consider a concurrent execution of $n>1$ different sequential processes subject to various constraints on their simultaneous access to shared resources.  
For example, these sequential processes might represent cash machines accessing a shared bank account, parallel threads of a computer saving files in a shared hard drive, or different logical circuits in an asynchronous microprocessor.    
In order to protect those shared resources against hazardous, simultaneous access by different processes, each sequential process is often required to obtain some sort of lock for the resource before performing some operation on the resource.  
The outcome of a concurrent execution is generally sensitive to the relative speeds at which those different sequential processes acquire and release their locks.   
The state stream $X$ can be constructed as a substream of the directed hypercube $\I^{(n,0)}$ containing the extrema of $\I^{(n,0)}$.
Intuitively, $\I^{(n,0)}-X$ consists of all states that are forbidden due to locks.
The substream $X_k\subset X$, of all tuples $(t_1,\ldots,t_n)\in X$ for which at least $n-k$ of the coordinates take values in $0,1$, represents the states of a restricted program that only runs at most $k$ different processes at the same time.  
An execution can be identified with a \textit{dipath} on $X$, a stream map $\I^{(1,0)}\ra X$.
We assume each sequential process can run unimpeded if each of the processses has either not started or not completed; equivalently take $X_1=(\I^n)_1$.
We further assume our constraints on parallelism are simple in the sense that $\I^{(n,0)}-X$ is a union of open isothetic hyperrectangles.  

\subsection{Serializability}
One desired property of a concurrent program is \textit{serializability}, an equivalence in outputs between each possible execution and an execution that operates each sequential process one at a time in some order.  
Serializability admits an intuitive geometric interpretation.
Two executions yield identical discrete outputs if their corresponding dipaths are h-homotopic relative endpoints.  
The concurrent program is thus serializable if and only if each extrema-preserving dipath on $X$ is h-homotopic relative endpoints to a dipath with image in $X_1$.  
Even though $X_1\ira X$ is almost never an h-cofibration [Theorem \ref{thm:NDR}], serializability can be shown to be equivalent to the condition that $\pi_{1,0}(X/X_1)=0$. 
This dihomotopical condition is generally more subtle than its classical homotopical counterpart, the condition that $\pi_{0,1}(X/X_1)=0$.

\subsection{2PL}
\textit{Two-phase locking} (2PL) is a protocol that requires each sequential process in a concurrent program goes through two phases in order: in the \textit{expansion phase}, each sequential process acquires all locks that it needs; in the \textit{shrinking phase} phase, each sequential process releases all locks that it had acquired.  
A standard result in formal verification is the following.

\begin{prop*}
  A concurrent computer program that follows 2PL is serializable.  
\end{prop*}

Proofs are interpretable in dihomotopical terms \cite{fajstrup2006algebraic,gunawardena2001homotopy}. 
2PL imposes strong constraints on the geometry of the forbidden region $\I^{(n,0)}-X$.
Those constraints make it then possible to h-homotope an extrema-preserving dipath on $X$ so that its image lies in $X_1$.  
The construction of the requisite homotopies is delicate and combinatorial \cite{fajstrup2006algebraic}, requiring suitable cubical subdivisions of the state stream $X$ as well as dipaths on $X$ \cite[Theorem 4.1]{fajstrup2005dipaths}.  
However, these point-set constructions can be encapsulated into a diagram chase that illustrates the general strategy.  
It suffices to show that $\pi_{1,0}(X/X_{n-1})=0$ under 2PL by induction on the number $n$ of processes.  
Consider the commutative diagram
\begin{equation*}
  \begin{tikzcd}
		\pi_{0,2}(\sphere^{(n,0)})\ar[r] & \pi_{0,2}(\sphere^{(n,0)},X/X_{n-1})\ar[r] & \pi_{0,1}(X/X_{n-1})\ar[r] & \pi_{0,1}\sphere^{(n,0)}\\
		\pi_{1,1}(\sphere^{(n,0)})\ar[r]\ar[u] & \pi_{1,1}(\sphere^{(n,0)},X/X_{n-1})\ar[r]\ar[u] & \pi_{1,0}(X/X_{n-1})\ar[r]\ar[u] & \pi_{1,0}\sphere^{(n,0)}\ar[u]
  \\
	\pi_{2,0}(\sphere^{(n,0)})\ar[u]\ar[r] & \pi_{2,0}^-(\sphere^{(n,0)},X/X_{n-1})\ar[u]\ar[r] & \pi_{1,0}(X/X_{n-1})\ar[u,equals]\ar[r] & \pi_{1,0}\sphere^{(n,0)}
	\ar[u,equals]
	\end{tikzcd}
\end{equation*}
whose top and middle rows are long exact sequences induced by based h-fiber sequences [Example \ref{eg:les}], vertical arrows are degree comparison maps, and the bottom row is a chain complex defined as an analogue of the higher rows but with Moore paths replaced by Moore diapths in the construction [Example \ref{eg:di-les}].  
  
\subsection{Two processes}
The case $n=2$ reduces to classical homotopy theory.
In this case, $X$ is non-positively curved in a suitably directed sense \cite[Definition 1.28]{mimram2020directed} and hence the arrow $\pi_{1,0}(X/X_{n-1})\ra\pi_{0,1}(X/X_{n-1})$ above is injective on monoid generators \cite[Theorem 2.30]{mimram2020directed}.
The forbidden region $\sphere^{(2,0)}-X/X_1=\I^{(2,0)}-X$ is path-connected under 2PL \cite[p3]{gunawardena2001homotopy}.
Alexander Duality implies $H_1(X/X_1;\Z)$ vanishes.
The Hurewicz Theorem implies $\pi_{0,1}(X/X_1)$, the fundamental group of a bouquet of spheres and hence free, vanishes.
Hence $\pi_{1,0}(X/X_1)$ vanishes.  

\subsection{More processes}
The case $n>2$ is more subtle. 
We have $\pi_{1,0}\sphere^{(n,0)}=0$ by directed simplicial approximation \cite[Theorem 8.14]{krishnan2015cubical}.
The middle horizontal arrow in the middle row is surjective by exactness.
It suffices to show $\pi_{1,1}\sphere^{(n,0)}\ra\pi_{1,1}(\sphere^{(n,0)},X/X_{n-1})$ is surjective by exactness.
Therefore consider an element
$$[\gamma]\in\pi_{1,1}(\sphere^{(n,0)},X/X_{n-1}),$$
represented by a Moore path $\gamma$ in $\Omega^{(1,0)}_\star\sphere^{(n,0)}$, starting with the basepoint and ending with a based map taking values in $X/X_{n-1}$.  
We can take $\gamma$ to lift to a Moore path $\tilde\gamma$ of dipaths on $\I^{(n,0)}$ without loss of generality by an application of directed simplicial approximation \cite[Theorems 7.1, 8.14]{krishnan2015cubical}.  
We can assume $\tilde\gamma=\gamma_1\gamma_2^{\dagger}\cdots\gamma_{2k}^{\dagger}\gamma_{2k+1}$ for choices $\gamma_1,\gamma_2,\ldots,\gamma_{2k+1}$ of Moore dipaths [Example \ref{eg:moore.dipaths}], where $\zeta^\dagger$ denotes the reverse of a Moore path $\zeta$ by $\sphere^{(1,0)}$ compact and $\I^{(n,0)}$ realizable as an edge-oriented cubical complex \cite[Theorem 8.14]{krishnan2015cubical}.  
If $k>0$, it is possible to use the lattice structure on $\I^{(n,0)}$ to construct a Moore dipath $\delta$ that ends in a dipath having image in $\partial\I^{(n,0)}$, so that $\gamma_1\delta$ represents an element in $\pi_{1,1}(\sphere^{(n,0)})$.
Without changing $[\gamma]$, we can replace $\gamma_1$ with $\gamma_1\delta$ and $\gamma_2$ with $\gamma_2\delta$.
Thus without changing the image of $[\gamma]$ in the cokernel of the left arrow in the middle row, we can take $k=0$ by an inductive argument.  
In particular, we can take $[\gamma]$ to come from $\pi_{2,0}^-(\sphere^{(n,0)},X/X_{n-1})$ without loss of generality.

The heart of the argument is to show that $[\gamma]$ further comes from $\pi_{2,0}\sphere^{(n,0)}$.
It suffices to take $\tilde\gamma$ to pass through the forbidden region; otherwise $\tilde\gamma$ can be replaced with $\tilde\gamma^{-1}\tilde\gamma$ and hence a constant Moore dipath.  
Then $\mathrm{cod}_{\Tau \I^{(n,0)}}\tilde\gamma$ represents an execution which, beyond a certain time, occurs in the shrinking phase for all the processors (cf. \cite[Lemma  7.4]{fajstrup2006algebraic}) under 2PL. 
It can then be shown that $\tilde\gamma$ itself can be replaced, without loss of generality, by its extension to a Moore dipath whose codomain has image in $\partial\I^n$ by the special geometry of the forbidden region \cite[Lemma  7.4]{fajstrup2006algebraic}.
In other words, $\gamma$ represents an element in $\pi_{2,0}\sphere^{(n,0)}$. 

\section{Conclusion}\label{sec:conclusion}
The goal is to extend these kinds of diagram chases for more computable invariants. 
The classifying stream construction [Example \ref{eg:classifying.streams}] extends to an endofunctor on commutative monoid objects in $\STREAMS$.  
Iterations of the based $p$-fold classifying stream construction with the $q$-fold classifying space construction on discrete commutative monoids represent a bigraded monoid-valued refinement $\tilde{H}^{p,q}((X,x);M)$ of ordinary reduced cohomology on a based stream $(X,x)$, taking coefficients in a discrete commutative monoid $M$.
Based h-cofiber sequences in $\STREAMS_\star$ then induce long exact sequences of such directed cohomology monoids.
On streams of interest in state space analysis, directed simplicial approximation \cite[Theorems 7.1, 8.14]{krishnan2015cubical} can be used to prove an excisive property for $\tilde{H}^{p,q}$.
We hope such excisive properties can greatly simplify the kinds of arguments about relative dihomotopy demonstrated in the previous section.  

\section{Acknowledgements}
This work was supported by AFOSR grant FA9550-16-1-0212.
The authors thank the anonymous refeee and the editor for helpful comments, corrections, and suggestions.  

\appendix
\addcontentsline{toc}{section}{Appendix}
\addtocontents{toc}{\protect\setcounter{tocdepth}{0}}


\bibliography{h-streams}{}
\bibliographystyle{plain}
\end{document}